\documentclass[reqno]{amsart}
\usepackage{mathrsfs}
\usepackage{color}
\usepackage{amsmath}
\usepackage{amsfonts}
\usepackage{amssymb}
\usepackage{graphicx}%
\usepackage{hyperref}
\usepackage{bbm}

 \newtheorem{Theorem}{Theorem}[section]
 \newtheorem{Corollary}[Theorem]{Corollary}
 \newtheorem{Lemma}[Theorem]{Lemma}
 \newtheorem{Proposition}[Theorem]{Proposition}

 \newtheorem{Remark}[Theorem]{Remark}

 \numberwithin{equation}{section}

\begin{document}

\title[the weighted $L^2$ integration]
{a support function of the weighted $L^2$ integrations on the superlevel sets of the weights}

\date{\today}
\author[Qi'an Guan]{Qi'an Guan$^{*}$}
\address{Qi'an Guan: School of Mathematical Sciences, and Beijing International Center for
Mathematical Research, Peking University, Beijing, 100871, China.}
\email{guanqian@math.pku.edu.cn}

\author[Zhenqian Li]{Zhenqian Li}
\address{Zhenqian Li:  School of Mathematical Sciences, Peking University, Beijing, 100871,
China.}
\email{lizhenqian@amss.ac.cn}

\author[Jiafu Ning]{Jiafu Ning$^{*}$}

\address{Jiafu Ning: College of Mathematics and Statistics, Chongqing University, Chongqing 401331,
 China.}
\email{jfning@cqu.edu.cn}

\thanks{The first author was partially supported by NSFC-11522101 and NSFC-11431013, the third author was partially supported by NSFC-11501058 and
NSFC-11501059.}



\thanks{$*$ The first author and the third author are both the corresponding authors.}


\begin{abstract}
In this article,
we establish a support function of the weighted $L^2$ integrations
on the superlevel sets of the weights with optimal asymptoticity near the positive infinity,
which is an analogue of the truth of Demailly's strong openness conjecture.
\end{abstract}

\maketitle

\section{Introduction}

Let $X$ be a domain in $\mathbb{C}^n$, and
let $\varphi$ be a negative plurisubharmonic function on $X$.
The multiplier ideal sheaf $\mathscr{I}(\varphi)$ can be defined
as the sheaf of germs of holomorphic functions $f$ such that
$|f|^{2}e^{-\varphi}$ is locally integrable (see \cite{D-K01}, \cite{Nadel90}).
It's well-known that $\mathscr{I}(\varphi)$ is a coherent analytic sheaf (see \cite{demailly2010}).
Let $\mathscr{I}_{+}(\varphi):=\cup_{\varepsilon>0}\mathscr{I}((1+\varepsilon)\varphi).$
In \cite{GZopen}, Guan and Zhou proved Demailly's strong openness conjecture, i.e., $\mathscr{I}_{+}(\varphi)=\mathscr{I}(\varphi).$

Let $z_{0}$ be a point in a pseudoconvex domain $D\subset X$,
and let $F$ be a holomorphic function near $z_{0}$.
Let $D_t=\{z\in D\big|\varphi\geq-t\}$ the superlevel set of weight $\varphi$,
and let $C_{F,\varphi,t}(z_{0})$ be the infimum of $\int_{D_t}|F_{1}|^{2}d\lambda_{n}$ for all $F_{1}\in\mathcal{O}(D)$ satisfying condition $(F_{1}-F,z_{0})\in\mathscr{I}(\varphi)_{z_{0}}$,
where $d\lambda_{n}$ is
the  Lebesgue measure on $\mathbb{C}^n$.

When $C_{F,\varphi,t}(z_{0})=0$ or $+\infty$, set $\frac{\int_{D_t}|F|^{2}e^{-\varphi}d\lambda_n}{C_{F,\varphi,t}(z_{0})}=+\infty$.
Then the truth of the strong openness conjecture
is equivalent to the statement that
for any pseudoconvex domain $D\subset X$, any point $z_{0}\in D$ and any holomorphic $F$ on $D$,
if
$\lim_{t\to+\infty}\frac{\int_{D_t}|F|^{2}e^{-(1+\varepsilon)\varphi}d\lambda_n}{C_{F,(1+\varepsilon)\varphi,t}(z_{0})}$
$=+\infty$
holds for any $\varepsilon>0,$
then
\begin{equation}
\label{equ:0729c}
\lim_{t\to+\infty}(\lim_{\varepsilon\to0^+}\frac{\int_{D_t}|F|^{2}e^{-(1+\varepsilon)\varphi}d\lambda_n}{C_{F,(1+\varepsilon)\varphi,t}(z_{0})})=+\infty
\end{equation}
(for details see Section \ref{sec:reformulation}).

In the present article,
we establish the following support function of
$\frac{\int_{D_t}|F|^{2}e^{-\varphi}d\lambda_n}{C_{F,\varphi,t}(z_{0})}$ (independent of $D$, $\varphi$ and $F$) with optimal asymptoticity.
\begin{Theorem}\label{thm 1}
For any pseudoconvex domain $D$ in $\mathbb{C}^{n}$, and any holomorphic function $F$ and negative plurisubharmonic function $\varphi$ on $D$,
\begin{equation}
\label{equ:20170430}
\frac{\int_{D_t}|F|^{2}e^{-\varphi}d\lambda_n}{C_{F,\varphi,t}(z_{0})}\geq h^{-1}(t-3)
\end{equation}
holds
for any $t\in(0,+\infty)$, where $h(x)=x+\log x$, and $h^{-1}(t-3)$ is the support function.
\end{Theorem}

The following remark gives the optimal asymptoticity near $+\infty$ of the support function $h^{-1}(t-3)$ in inequality \ref{equ:20170430}.

\begin{Remark}
Take $D=\Delta\subset\mathbb{C}$, $z_{0}=o$ the origin of $\mathbb{C}$, $F\equiv1$ and $\varphi=2\log|z|$.
Note that $C_{F,\varphi,t}(z_{0})=\pi(1-e^{-t})$ and $\int_{D_t}|F|^{2}e^{-\varphi}d\lambda_n=t\pi$,
then $\frac{\int_{D_t}|F|^{2}e^{-\varphi}d\lambda_n}{C_{F,\varphi,t}(z_{0})}=\frac{t}{1-e^{-t}}$.
One can obtain $\lim_{t\to+\infty}\frac{h^{-1}(t-3)}{\frac{t}{1-e^{-t}}}=1$,
which implies the optimal asymptoticity of the support function $h^{-1}(t-3)$ in inequality \ref{equ:20170430}, when $t$ goes to $+\infty$.
\end{Remark}

By equality \ref{equ:0729c},
it follows that Theorem \ref{thm 1} implies the truth of Demailly's strong openness conjecture.

\begin{Corollary}\cite{GZopen}
$\mathscr{I}_{+}(\varphi)=\mathscr{I}(\varphi).$
\end{Corollary}

\section{Preparations}

\subsection{The equivalent statement of the truth of Demailly's strong openness conjecture}\label{sec:reformulation}

By the dominated convergence theorem, it follows that
\begin{equation}
\label{equ:0729a}
\lim_{\varepsilon\to0^+}\int_{D_t}|F|^{2}e^{-(1+\varepsilon)\varphi}d\lambda_{n}=\int_{D_t}|F|^{2}e^{-\varphi}d\lambda_{n}.
\end{equation}
By the definition of $C_{F,(1+\varepsilon)\varphi,t}(z_{0})$,
it follows that $\lim_{\varepsilon\to0^+}C_{F,(1+\varepsilon)\varphi,t}(z_{0})=$the infimum of $\int_{D_t}|F_{1}|^{2}d\lambda_{n}$ for all $F_{1}\in\mathcal{O}(D)$ satisfying condition $(F_{1}-F,z_{0})\in\mathscr{I}_{+}(\varphi)_{z_{0}}$.
Then it suffices to consider that equality \ref{equ:0729c} implies the truth of Demailly's strong openness conjecture.

If not, i.e., there exists $z_{0}\in X$ satisfying $\mathscr{I}_{+}(\varphi)_{z_{0}}\subsetneq\mathscr{I}(\varphi)_{z_{0}}$,
then there exists holomorphic function $F$ near $z_{0}$ such that
$(F,z_{0})\in\mathscr{I}(\varphi)_{z_{0}}$ and $(F,z_{0})\not\in\mathscr{I}_{+}(\varphi)_{z_{0}}$.
Choosing $D$ small enough, it follows from $(F,z_{0})\in\mathscr{I}(\varphi)_{z_{0}}$ that $\int_{D}|F|^{2}e^{-\varphi}d\lambda_{n}<+\infty$,
which implies
\begin{equation}
\label{equ:0729b}\lim_{t\to+\infty}\int_{D_t}|F|^{2}e^{-\varphi}d\lambda_{n}=\int_{D}|F|^{2}e^{-\varphi}d\lambda_{n}<+\infty.
\end{equation}
Note that $C_{F,(1+\varepsilon)\varphi,t}(z_{0})$ is upper bounded with respect to $t$ for $D$ small enough,
it follows that
$$\lim_{t\to+\infty}\frac{\int_{D_t}|F|^{2}e^{-(1+\varepsilon)\varphi}d\lambda_n}{C_{F,(1+\varepsilon)\varphi,t}(z_{0})}=+\infty$$
for any $\varepsilon>0$.
As $(F,z_{0})\not\in\mathscr{I}_{+}(\varphi)_{z_{0}}$,
it follows that the infimum of $\int_{D_t}|F_{1}|^{2}d\lambda_{n}$ for all $F_{1}\in\mathcal{O}(D)$ satisfying condition $(F_{1}-F,z_{0})\in\mathscr{I}_{+}(\varphi)_{z_{0}}$ is bigger than a positive constant $C_{0}$ for $t$ large enough.
Combining with inequality \ref{equ:0729b},
one can obtain that
$$\lim_{t\to+\infty}(\lim_{\varepsilon\to0^+}\frac{\int_{D_t}|F|^{2}e^{-(1+\varepsilon)\varphi}d\lambda_n}{C_{F,(1+\varepsilon)\varphi,t}(z_{0})})
\leq\lim_{t\to+\infty}\frac{\int_{D_t}|F|^{2}e^{-\varphi}d\lambda_n}{C_{0}}
=\frac{\int_{D}|F|^{2}e^{-\varphi}d\lambda_n}{C_{0}}<+\infty,$$
which contradicts equality \ref{equ:0729c}.
Then equality \ref{equ:0729c} implies the truth of Demailly's strong openness conjecture.

\subsection{$\bar\partial-$equation with $L^{2}$ estimates}

We prove Proposition \ref{p:effect} by the following Lemma, whose various forms have already appeared in \cite{guan-zhou13p,GZeffective}.

\begin{Lemma} \label{p:GZ_JM_sharp}
Let $B_{0}\in(0,1]$ be arbitrarily given. Let $U_{v}$ be a strongly pseudoconvex domain relatively compact in pseudoconvex domain
$D\subseteq\mathbb{C}^n$ containing $o$. Let $F$ be a holomorphic function on $D$. Let $\varphi$ be a negative plurisubharmonic function on $D$, such that $u(o)=-\infty$. Then there exists a holomorphic function $F_{v,t_{0}}$ on $U_{v}$, such that,$$(F_{v,t_{0}}-F,o)\in\mathscr{I}(\varphi)_{o}$$ and
\begin{equation}
\label{equ:3.4}
\begin{split}
&\int_{ U_v}|F_{v,t_0}-(1-b_{t_0}(\varphi))F|^{2}d\lambda_{n}\\
\leq&(1-e^{-(t_{0}+B_{0})})\int_{U_v}\frac{1}{B_{0}}(\mathbbm{1}_{\{-t_{0}-B_{0}<t<-t_{0}\}}\circ \varphi)|F|^{2}e^{-\varphi}d
\lambda_{n},
\end{split}
\end{equation}
where $b_{t_{0}}(t)=\int_{-\infty}^{t}\frac{1}{B_{0}}\mathbbm{1}_{\{-t_{0}-B_{0}< s<-t_{0}\}}ds$, and $t_{0}$ is a positive number.
\end{Lemma}

\begin{Remark}
\label{r:GZ_JM_sharp}
Replacing the strong pseudoconvexity of $U_{v}$ by pseudoconvexity, and $o$ by $z_{0}\in U_{v}$ for any $v$,
Lemma \ref{p:GZ_JM_sharp} also holds. We may take $U_v\subset\subset U_{v+1}$ and $\cup U_v=D$. In the following, by approximation, we
may take $D$ instead of $U_v$.
\end{Remark}

\subsection{A useful proposition}

Inspired by the proof of the main result in \cite{GZeffective}
and making some modifications, 
we obtain the following proposition.

\begin{Proposition}
\label{p:effect}
Let $C_{1}$, $C_{2}$ be two positive constants and $t\geq\log \frac{C_1}{C_2}$. Let $p>0$. We consider the set of $(F,\varphi,p)$ satisfying
\\
$(1)$ $\int_{D_t}|F|^{2}e^{-\varphi}\leq C_{1}$;\qquad
$(2)$ $C_{F,p\varphi,t}(z_{0})\geq C_{2}$;\qquad\\
where $C_{F,p\varphi,t}(z_{0})$ is the infimum of $\int_{D_t}|F_{1}|^{2}d\lambda_{n}$ for all $F_{1}\in\mathcal{O}(D)$ satisfying condition $(F_{1}-F,z_{0})\in\mathscr{I}(p\varphi)_{z_{0}}$, and
$D_t=\{z\in D\big|\varphi\geq-t\}.$

Then, for $p\neq 1$ and $1/2$, we have
\begin{equation}
\label{equ:effect20140126.7}
\begin{split}
(1-\frac{1}{p})^{-1}(\frac{C_{2}}{C_{1}})((\frac{C_{2}}{C_{1}})^{p-1}-e^{-(p-1)t})-&2(1-\frac{1}{2p})^{-1}(\frac{C_{2}}{C_{1}})^{\frac{1}{2}}
((\frac{C_{2}}{C_{1}})^{p-\frac{1}{2}}-e^{-(p-\frac{1}{2})t})
\\ &+((\frac{C_{2}}{C_{1}})^{p}-e^{-pt})\leq1.
\end{split}
\end{equation}

For $p=1$, we have
\begin{equation}
\label{equ:effect20140126.9}
(\frac{C_{2}}{C_{1}})(t+\log \frac{C_{2}}{C_{1}})-4(\frac{C_{2}}{C_{1}})^{\frac{1}{2}}
((\frac{C_{2}}{C_{1}})^{\frac{1}{2}}-e^{-\frac{t}{2}})
+((\frac{C_{2}}{C_{1}})-e^{-t})\leq1.
\end{equation}

For $p=1/2$, we have
$$\frac{C_2}{C_1}(e^{\frac{t}{2}}-(\frac{C_2}{C_1})^{-\frac{1}{2}})-(\frac{C_2}{C_1})^{\frac{1}{2}}(t+\log \frac{C_2}{C_1})
+(\frac{C_2}{C_1})^{\frac{1}{2}}-e^{-\frac{t}{2}}\leq 1.$$
\end{Proposition}

\begin{proof}

As $C_{F,p\varphi,t}(z_{0})\geq C_{2}>0$,
then we obtain $\varphi(z_{0})=-\infty$.

Let $\psi:=p\varphi$.

Using Lemma \ref{p:GZ_JM_sharp} and Remark \ref{r:GZ_JM_sharp},
we obtain that
\begin{equation}
\label{equ:effect20140120.1}
\begin{split}
&((\int_{D_t}|F_{v,t_0}|^{2}d\lambda_{n})^{1/2}-(\int_{D_t}|(1-b_{t_0}(\psi))F|^{2}d\lambda_{n})^{1/2})^{2}
\\&\leq \int_{D_t}|F_{v,t_0}-(1-b_{t_0}(\psi))F|^{2}d\lambda_{n}
\\&\leq \int_{D}|F_{v,t_0}-(1-b_{t_0}(\psi))F|^{2}d\lambda_{n}
\\&\leq(1-e^{-(t_{0}+B_{0})})\int_{D}\frac{1}{B_{0}}(\mathbbm{1}_{\{-t_{0}-B_{0}<t<-t_{0}\}}\circ\psi)|F|^{2}e^{-\psi}d\lambda_{n}.
\end{split}
\end{equation}
Note that
\begin{equation}
\label{equ:effect20140122.2}
\begin{split}
&(1-e^{-(t_{0}+B_{0})})\int_{D}\frac{1}{B_{0}}(\mathbbm{1}_{\{-t_{0}-B_{0}<t<-t_{0}\}}\circ\psi)|F|^{2}e^{-\psi}d\lambda_{n}
\\&\leq(e^{t_{0}+B_{0}}-1)
\int_{D}\frac{1}{B_{0}}(\mathbbm{1}_{\{-t_{0}-B_{0}<t<-t_{0}\}}\circ\psi)|F|^{2}d\lambda_{n},
\end{split}
\end{equation}
Then we have
\begin{equation}
\label{equ:effect20140122.3}
\begin{split}
&((\int_{D_t}|F_{v,t_0}|^{2}d\lambda_{n})^{1/2}-(\int_{D_t}|(1-b_{t_0}(\psi))F|^{2}d\lambda_{n})^{1/2})^{2}
\\&\leq(e^{t_{0}+B_{0}}-1)
\int_{D}\frac{1}{B_{0}}(\mathbbm{1}_{\{-t_{0}-B_{0}<t<-t_{0}\}}\circ\psi)|F|^{2}d\lambda_{n},
\end{split}
\end{equation}
Note that
$$e^{\frac{1}{p}t_{0}}\int_{D_t}|(1-b_{t_0}(\psi))F|^{2}d\lambda_{n}\leq
e^{\frac{1}{p}t_{0}}\int_{D_t}|\mathbbm{1}_{\{\frac{1}{p}\psi\leq-\frac{1}{p}t_{0}\}}F|^{2}d\lambda_{n}\leq \int_{D_t}|F|^{2}e^{-\frac{1}{p}\psi}d\lambda_{n}$$
As
$$C_{1}\geq\int_{D_t}|F|^{2}e^{-\frac{1}{p}\psi}d\lambda_{n}=\int_{D_t}|F|^{2}e^{-\varphi}d\lambda_{n}.$$
and $\inf\{\int_{D_t}|F_{1}|^{2}d\lambda_{n}|F_{1}\in\mathcal{O}(D),(F_{1}-F,z_{0})\in\mathscr{I}(\psi)_{z_{0}}\}\geq C_{2}$,
when
$$e^{\frac{1}{p}t_{0}}\geq \frac{C_{1}}{C_{2}},$$
then we have
\begin{equation}
\label{equ:effect20140122.4}
\begin{split}
&(C_{2}^{1/2}-(C_{1}e^{-\frac{1}{p}t_{0}})^{1/2})^{2}
\\&\leq((\int_{D_t}|F_{v,t_0}|^{2}d\lambda_{n})^{1/2}-(\int_{D_t}|(1-b_{t_0}(\psi))F|^{2}d\lambda_{n})^{1/2})^{2}.
\end{split}
\end{equation}
It follows that
\begin{equation}
\label{equ:effect20140122.4}
\begin{split}
&(C_{2}^{1/2}-(C_{1}e^{-\frac{1}{p}t_{0}})^{1/2})^{2}
\\&\leq(e^{t_{0}+B_{0}}-1)
\int_{D}\frac{1}{B_{0}}(\mathbbm{1}_{\{-t_{0}-B_{0}<t<-t_{0}\}}\circ\psi)|F|^{2}d\lambda_{n}
\\&\leq e^{t_{0}+B_{0}}\int_{D}\frac{1}{B_{0}}(\mathbbm{1}_{\{-t_{0}-B_{0}<t<-t_{0}\}}\circ\psi)|F|^{2}d\lambda_{n}.
\end{split}
\end{equation}
Replacing $t_{0}$ by $kB_{0}$,
and assuming that $e^{\frac{1}{p}kB_{0}}\geq\frac{C_{1}}{C_{2}}$,
we obtain that
\begin{equation}
\label{equ:effect20140126.1}
\begin{split}
&(C_{2}^{1/2}-(C_{1}e^{-\frac{1}{p}kB_{0}})^{1/2})^{2}
\\&\leq e^{(k+1)B_{0}}\int_{D}\frac{1}{B_{0}}(\mathbbm{1}_{\{-(k+1)B_{0}<t<-kB_{0}\}}\circ\psi)|F|^{2}d\lambda_{n}.
\end{split}
\end{equation}
It follows that
\begin{equation}
\label{equ:effect20140126.3}
\begin{split}
& B_{0}e^{-(k+1)B_{0}}e^{\frac{1}{p}kB_{0}}(C_{2}^{1/2}-(C_{1}e^{-\frac{1}{p}kB_{0}})^{1/2})^{2}
\\&\leq e^{\frac{1}{p}kB_{0}}\int_{D}(\mathbbm{1}_{\{-(k+1)B_{0}<t<-kB_{0}\}}\circ\psi)|F|^{2}d\lambda_{n}
\\&\leq \int_{D}(\mathbbm{1}_{\{-(k+1)B_{0}/p<t<-kB_{0}/p\}}\circ\varphi)|F|^{2}e^{-\varphi}d\lambda_{n},
\end{split}
\end{equation}
the last inequality holds because $\psi=p\varphi$.

Taking $k_{0}$, such that
$$e^{\frac{1}{p}k_{0}B_{0}}\geq \frac{C_{1}}{C_{2}}\geq e^{\frac{1}{p}(k_{0}-1)B_{0}},$$
and taking sum, we obtain
\begin{equation}
\label{equ:effect20140126.4}
\begin{split}
&\sum_{k=k_{0}}^{[pt/B_0]-1}B_{0}e^{-(k+1)B_{0}}e^{\frac{1}{p}kB_{0}}(C_{2}^{1/2}-(C_{1}e^{-\frac{1}{p}kB_{0}})^{1/2})^{2}
\\&\leq \sum_{k=k_{0}}^{[pt/B_0]-1}\int_{D}(\mathbbm{1}_{\{-(k+1)B_{0}/p<t<-kB_{0}/p\}}\circ\varphi)|F|^{2}e^{-\varphi}d\lambda_{n}
\\&\leq\int_{D_t}|F|^{2}e^{-\varphi}d\lambda_{n}\leq C_{1}.
\end{split}
\end{equation}
Note that when $p>0$, $p\neq1$ and $p\neq 1/2$, we have
\begin{equation}
\label{equ:effect20140126.5}
\begin{split}
&\sum_{k=k_{0}}^{[pt/B_0]-1}B_{0}e^{-(k+1)B_{0}}e^{\frac{1}{p}kB_{0}}(C_{2}^{1/2}-(C_{1}e^{-\frac{1}{p}kB_{0}})^{1/2})^{2}
\\=&\sum_{k=k_{0}}^{[pt/B_0]-1}B_0e^{-B_0}(e^{(\frac{1}{p}-1)kB_0}C_{2}-
2e^{(\frac{1}{2p}-1)kB_{0}}C_{2}^{1/2}C_{1}^{1/2}+e^{-kB_{0}}C_{1})
\\=&B_0e^{-B_0}(C_2\frac{e^{-(1-\frac{1}{p})k_0B_{0}}-e^{-(1-\frac{1}{p})[\frac{pt}{B_0}]B_0}}{1-e^{-(1-\frac{1}{p})B_{0}}}
\\&-2C_{2}^{1/2}C_{1}^{1/2}\frac{e^{-k_{0}(1-\frac{1}{2p})B_{0}}-e^{-(1-\frac{1}{2p})[\frac{pt}{B_0}]B_0}}
{1-e^{-(1-\frac{1}{2p})B_{0}}}
+C_1\frac{e^{-k_0B_0}-e^{[\frac{pt}{B_0}]B_0}}{1-e^{-B_0}})\\
\geq&B_0e^{-B_0}(C_2\frac{e^{-(1-\frac{1}{p})B_0}(\frac{C_2}{C_1})^{p-1}-e^{-(1-\frac{1}{p})[\frac{pt}{B_0}]B_0}}
{1-e^{-(1-\frac{1}{p})B_{0}}}
\\&-2C_{2}^{1/2}C_{1}^{1/2}\frac{(\frac{C_2}{C_1})^{p(1-\frac{1}{2p})}-e^{-(1-\frac{1}{2p})[\frac{pt}{B_0}]B_0}}
{1-e^{-(1-\frac{1}{2p})B_{0}}}
+C_1\frac{e^{-B_0}(\frac{C_2}{C_1})^p-e^{[\frac{pt}{B_0}]B_0}} {1-e^{-B_0}}),
\end{split}
\end{equation}
Take limitation
\begin{equation}
\label{equ:effect20140126.6}
\begin{split}
\lim_{B_{0}\to0}&B_0e^{-B_0}(C_2\frac{e^{-(1-\frac{1}{p})B_0}(\frac{C_2}{C_1})^{p-1}-e^{-(1-\frac{1}{p})[\frac{pt}{B_0}]B_0}}
{1-e^{-(1-\frac{1}{p})B_{0}}}
\\&-2C_{2}^{1/2}C_{1}^{1/2}\frac{(\frac{C_2}{C_1})^{p(1-\frac{1}{2p})}-e^{-(1-\frac{1}{2p})[\frac{pt}{B_0}]B_0}}
{1-e^{-(1-\frac{1}{2p})B_{0}}}
+C_1\frac{e^{-B_0}(\frac{C_2}{C_1})^p-e^{[\frac{pt}{B_0}]B_0}} {1-e^{-B_0}})
\\=&C_2\frac{(\frac{C_2}{C_1})^{p-1}-e^{-(1-\frac{1}{p})pt}}{1-\frac{1}{p}}
\\&-2C_{2}^{1/2}C_{1}^{1/2}\frac{(\frac{C_{2}}{C_{1}})^{p(1-\frac{1}{2p})}-e^{-(1-\frac{1}{2p})pt}}{1-\frac{1}{2p}}
+C_1((\frac{C_{2}}{C_{1}})^{p}-e^{-pt})
\\=&C_2\frac{(\frac{C_2}{C_1})^{p-1}-e^{-(p-1)t}}{1-\frac{1}{p}}
\\&-2C_{2}^{1/2}C_{1}^{1/2}\frac{(\frac{C_{2}}{C_{1}})^{p-\frac{1}{2}}-e^{-(p-\frac{1}{2})t}}{1-\frac{1}{2p}}
+C_1((\frac{C_{2}}{C_{1}})^{p}-e^{-pt}).
\end{split}
\end{equation}
By inequality \ref{equ:effect20140126.4}, \ref{equ:effect20140126.5} and \ref{equ:effect20140126.6},
it follows that
\begin{equation}
\label{equ:effect20140126.8}
\begin{split}
(1-\frac{1}{p})^{-1}(\frac{C_{2}}{C_{1}})((\frac{C_{2}}{C_{1}})^{p-1}-&e^{-(p-1)t})-2(1-\frac{1}{2p})^{-1}(\frac{C_{2}}{C_{1}})^{\frac{1}{2}}
((\frac{C_{2}}{C_{1}})^{p-\frac{1}{2}}-e^{-(p-\frac{1}{2})t})
\\&+((\frac{C_{2}}{C_{1}})^{p}-e^{-pt})\leq1.
\end{split}
\end{equation}
When $p=1$, we can take $p=1$ to \eqref{equ:effect20140126.4}, by the same way, with some changes, and we can get \eqref{equ:effect20140126.9}.
We write down the details in the following.

Taking $k_{0}$, such that
$$e^{k_{0}B_{0}}\geq \frac{C_{1}}{C_{2}}\geq e^{(k_{0}-1)B_{0}}.$$
Let $p=1$ in \eqref{equ:effect20140126.3}, and take sum, then we have
\begin{equation}
\label{eq5}
\begin{split}
&\sum_{k=k_{0}}^{[t/B_0]-1}B_{0}e^{-(k+1)B_{0}}e^{kB_{0}}(C_{2}^{1/2}-(C_{1}e^{-kB_{0}})^{1/2})^{2}
\\&\leq \sum_{k=k_{0}}^{[t/B_0]-1}\int_{D}(\mathbbm{1}_{\{-(k+1)B_{0}<t<-kB_{0}\}}\circ\varphi)|F|^{2}e^{-\varphi}d\lambda_{n}
\\&\leq\int_{D_t}|F|^{2}e^{-\varphi}d\lambda_{n}\leq C_{1}.
\end{split}
\end{equation}

It follows that
\begin{equation}
\label{eq6}
\begin{split}
&\sum_{k=k_{0}}^{[t/B_0]-1}B_{0}e^{-(k+1)B_{0}}e^{kB_{0}}(C_{2}^{1/2}-(C_{1}e^{-kB_{0}})^{1/2})^{2}
\\=&\sum_{k=k_{0}}^{[t/B_0]-1}B_0e^{-B_0}(C_{2}-
2e^{-\frac{1}{2}kB_{0}}C_{2}^{1/2}C_{1}^{1/2}+e^{-kB_{0}}C_{1})
\\=&B_0e^{-B_0}(C_2([\frac{t}{B_0}]-k_0)
\\&-2C_{2}^{1/2}C_{1}^{1/2}\frac{e^{-\frac{1}{2}k_{0}B_{0}}-e^{-\frac{1}{2}[\frac{t}{B_0}]B_0}}
{1-e^{-\frac{1}{2}B_{0}}}
+C_1\frac{e^{-k_0B_0}-e^{-[\frac{t}{B_0}]B_0}}{1-e^{-B_0}})\\
\geq&B_0e^{-B_0}(C_2([\frac{t}{B_0}]+\frac{1}{B_0}\log \frac{C_2}{C_1}-1)
\\&-2C_{2}^{1/2}C_{1}^{1/2}\frac{(\frac{C_2}{C_1})^{\frac{1}{2}}-e^{-\frac{1}{2}[\frac{t}{B_0}]B_0}}
{1-e^{-\frac{1}{2}B_{0}}}
+C_1\frac{e^{-B_0}\frac{C_2}{C_1}-e^{-[\frac{t}{B_0}]B_0}} {1-e^{-B_0}}).
\end{split}
\end{equation}

Take limitation
\begin{equation}
\label{eq7}
\begin{split}
\lim_{B_{0}\to0}&B_0e^{-B_0}(C_2([\frac{t}{B_0}]+\frac{1}{B_0}\log \frac{C_2}{C_1}-1)
\\&-2C_{2}^{1/2}C_{1}^{1/2}\frac{(\frac{C_2}{C_1})^{\frac{1}{2}}-e^{-\frac{1}{2}[\frac{t}{B_0}]B_0}}
{1-e^{-\frac{1}{2}B_{0}}}
+C_1\frac{e^{-B_0}\frac{C_2}{C_1}-e^{-[\frac{t}{B_0}]B_0}} {1-e^{-B_0}})
\\=&C_2(C_2t+\log \frac{C_2}{C_1})
\\&-4C_{2}^{1/2}C_{1}^{1/2}((\frac{C_{2}}{C_{1}})^{\frac{1}{2}}-e^{-\frac{1}{2}pt})
+C_1(\frac{C_{2}}{C_{1}}-e^{-t}).
\end{split}
\end{equation}
By inequalities \eqref{eq5} \eqref{eq6} and \eqref{eq7}, we have
$$(\frac{C_{2}}{C_{1}})(t+\log \frac{C_{2}}{C_{1}})-4(\frac{C_{2}}{C_{1}})^{\frac{1}{2}}
((\frac{C_{2}}{C_{1}})^{\frac{1}{2}}-e^{-\frac{t}{2}})
+((\frac{C_{2}}{C_{1}})-e^{-t})\leq1.$$

When $p=1/2$, take $k_{0}$, such that
$$e^{2k_0B_{0}}\geq \frac{C_{1}}{C_{2}}\geq e^{2(k_{0}-1)B_{0}}.$$
Let $p=1/2$ in \eqref{equ:effect20140126.3}, and take sum, then we have
\begin{equation}
\label{eq8}
\begin{split}
&\sum_{k=k_{0}}^{[\frac{t}{2B_0}]-1}B_{0}e^{-(k+1)B_{0}}e^{2kB_{0}}(C_{2}^{1/2}-(C_{1}e^{-2kB_{0}})^{1/2})^{2}
\\&\leq \sum_{k=k_{0}}^{[\frac{t}{2B_0}]-1}\int_{D}(\mathbbm{1}_{\{-2(k+1)B_{0}<t<-2kB_{0}\}}\circ\varphi)|F|^{2}e^{-\varphi}d\lambda_{n}
\\&\leq\int_{D_t}|F|^{2}e^{-\varphi}d\lambda_{n}\leq C_{1}.
\end{split}
\end{equation}
As

\begin{equation}\label{eq9}
\begin{split}
&\sum_{k=k_{0}}^{[\frac{t}{2B_0}]-1}B_{0}e^{-(k+1)B_{0}}e^{2kB_{0}}(C_{2}^{1/2}-(C_{1}e^{-2kB_{0}})^{1/2})^{2}
\\=&\sum_{k=k_{0}}^{[\frac{t}{2B_0}]-1}B_0e^{-B_0}(e^{kB_0}C_2-C_2^{\frac{1}{2}}C_1^{\frac{1}{2}}+C_1e^{-kB_0})
\\=&B_0e^{-B_0}(C_2\frac{e^{B_0}(\frac{C_2}{C_1})^{-\frac{1}{2}}-e^{[\frac{t}{2B_0}]B_0}}{1-e^{B_0}}
\\&-2C_2^{\frac{1}{2}}C_1^{\frac{1}{2}}([\frac{t}{2B_0}]-k_0)
+C_1\frac{e^{-k_0B_0-e^{-[\frac{t}{2B_0}]B_0}}}{1-e^{-B_0}})
\\\geq &B_0e^{-B_0}(C_2\frac{e^{B_0}(\frac{C_2}{C_1})^{\frac{1}{2}}-e^{[\frac{t}{2B_0}]B_0}}{1-e^{B_0}}
\\&-2C_2^{\frac{1}{2}}C_1^{\frac{1}{2}}([\frac{t}{2B_0}]+\frac{1}{2}\log \frac{C_2}{C_1})
+C_1\frac{e^{-B_0}(\frac{C_2}{C_1})^{\frac{1}{2}}-e^{-[\frac{t}{2B_0}]B_0}}{1-e^{-B_0}}).
\end{split}
\end{equation}
Take limitation
\begin{equation}\label{eq10}
\begin{split}
  \lim_{B_0\rightarrow 0} &B_0e^{-B_0}(C_2\frac{e^{B_0}(\frac{C_2}{C_1})^{\frac{1}{2}}-e^{[\frac{t}{2B_0}]B_0}}{1-e^{B_0}}
\\&-2C_2^{\frac{1}{2}}C_1^{\frac{1}{2}}([\frac{t}{2B_0}]+\frac{1}{2}\log \frac{C_2}{C_1})
+C_1\frac{e^{-B_0}(\frac{C_2}{C_1})^{\frac{1}{2}}-e^{-[\frac{t}{2B_0}]B_0}}{1-e^{-B_0}})
\\=&C_2(e^{\frac{t}{2}}-(\frac{C_2}{C_1})^{-\frac{1}{2}})-4C_2^{\frac{1}{2}}C_1^{\frac{1}{2}}(t+\log \frac{C_2}{C_1})
+C_1((\frac{C_2}{C_1})^{\frac{1}{2}}-e^{-\frac{t}{2}}).
\end{split}
\end{equation}
Combine inequalities \eqref{eq8} \eqref{eq9} and \eqref{eq10}, then we have
$$\frac{C_2}{C_1}(e^{\frac{t}{2}}-(\frac{C_2}{C_1})^{-\frac{1}{2}})-(\frac{C_2}{C_1})^{\frac{1}{2}}(t+\log \frac{C_2}{C_1})
+(\frac{C_2}{C_1})^{\frac{1}{2}}-e^{-\frac{t}{2}}\leq 1.$$
\end{proof}

\begin{Remark}
For $p=1$, let $p_j>1$, $j=1,2,\cdots$ and $lim_{j\rightarrow \infty}p_j=1$,
then from $C_{F,p_j\varphi,t}\geq C_{F,\varphi,t}\geq C_2$.
Therefore, inequality \eqref{equ:effect20140126.7} with $p=p_j$.
Taking $j\rightarrow \infty$, we can get
$$(\frac{C_{2}}{C_{1}})(t+\log \frac{C_{2}}{C_{1}})-4(\frac{C_{2}}{C_{1}})^{\frac{1}{2}}
((\frac{C_{2}}{C_{1}})^{\frac{1}{2}}-e^{-\frac{t}{2}})
+((\frac{C_{2}}{C_{1}})-e^{-t})\leq1.$$

For $p=1/2$, let $p_j>1/2$, $j=1,2,\cdots$ and $\lim_{j\rightarrow \infty}p_j=1/2$,
then from $C_{F,p_j\varphi,t}\geq C_{F,\frac{1}{2}\varphi,t}\geq C_2$.
Therefore, inequality \eqref{equ:effect20140126.7} with $p=p_j$.
Taking $j\rightarrow \infty$, we can get
$$\frac{C_2}{C_1}(e^{\frac{t}{2}}-(\frac{C_2}{C_1})^{-\frac{1}{2}})-(\frac{C_2}{C_1})^{\frac{1}{2}}(t+\log t)
+(\frac{C_2}{C_1})^{\frac{1}{2}}-e^{-\frac{t}{2}}\leq 1.$$
\end{Remark}

\section{Proof of Theorem \ref{thm 1}}

It suffices to prove the following proposition

\begin{Proposition}\label{prop 1}
Let $C_{1}$, $C_{2}$ be two positive constants and $t>\log\frac{C_1}{C_2}$. If the set of $(F,\varphi)$ satisfying\\
$(1)$ $\int_{D_t}|F|^{2}e^{-\varphi}=C_{1}$;\qquad
$(2)$ $C_{F,\varphi,t}(z_{0})=C_{2}$;\\
Then $\frac{\int_{D_t}|F|^{2}e^{-\varphi}d\lambda_n}{C_{F,\varphi,t}(z_{0})}\geq h^{-1}(t-3)$, where $h(x)=x+\log x$.
\end{Proposition}

Note that

(1) when $t>\log \frac{C_1}{C_2}$ and $C_{F,\varphi,t}(z_{0})>0$, Proposition \ref{prop 1} implies $\frac{\int_{D_t}|F|^{2}e^{-\varphi}d\lambda_n}{C_{F,\varphi,t}(z_{0})}\geq h^{-1}(t-3)$;

(2) when $t\leq \log \frac{C_1}{C_2}$ and $C_{F,\varphi,t}(z_{0})>0$, it follows that $\frac{\int_{D_t}|F|^{2}e^{-\varphi}d\lambda_n}{C_{F,\varphi,t}(z_{0})}=\frac{C_{1}}{C_{2}}\geq e^{t}\geq h^{-1}(t-3)$,

(3) when $C_{F,\varphi,t}(z_{0})=0$, set $\frac{\int_{D_t}|F|^{2}e^{-\varphi}d\lambda_n}{C_{F,\varphi,t}(z_{0})}=+\infty$,
\\
then Proposition \ref{prop 1} implies Theorem \ref{thm 1}.

\begin{proof}(Proof of Proposition \ref{prop 1})
By Proposition \ref{p:effect} (case $p=1$), we have
\begin{equation}\label{eq1}
(\frac{C_{2}}{C_{1}})(t+\log \frac{C_{2}}{C_{1}})-4(\frac{C_{2}}{C_{1}})^{\frac{1}{2}}
((\frac{C_{2}}{C_{1}})^{\frac{1}{2}}-e^{-\frac{t}{2}})
+((\frac{C_{2}}{C_{1}})-e^{-t})\leq1,
\end{equation}

Since
\begin{equation}\label{eq2}
\begin{split}
&(\frac{C_{2}}{C_{1}})(t+\log \frac{C_{2}}{C_{1}})-4(\frac{C_{2}}{C_{1}})^{\frac{1}{2}}
((\frac{C_{2}}{C_{1}})^{\frac{1}{2}}-e^{-\frac{t}{2}})+((\frac{C_{2}}{C_{1}})-e^{-t})\\
\geq &\frac{C_2}{C_1}(t+\log \frac{C_2}{C_1})-3\frac{C_2}{C_1},
\end{split}
\end{equation}
combining \eqref{eq1} and \eqref{eq2}, we get
\begin{equation}\label{eq3}
\frac{C_2}{C_1}(t+\log \frac{C_2}{C_1})-3\frac{C_2}{C_1}\leq 1,
\end{equation}
hence
$$\frac{C_1}{C_2}+\log \frac{C_1}{C_2}\geq t-3.$$
Set $h(x)=x+\log x$ for $x>0$, it is easy that $h(x)$ is increasing.
By \eqref{eq3}, we get $\frac{C_1}{C_2}\geq h^{-1}(t-3)$.
We may take $C_1=\int_{D_t}|F|^2e^{-\varphi}$, then the proposition follows.
\end{proof}

\bibliographystyle{references}
\bibliography{xbib}

\end{document}